\documentclass[12pt]{article}

\usepackage{amsmath, epsfig, cite, setspace}
\usepackage{amssymb}
\usepackage{amsfonts}
\usepackage{latexsym}
\usepackage{amsthm}

\makeatletter
\renewcommand{\@seccntformat}[1]{{\csname the#1\endcsname}.\hspace{.5em}}
\makeatother

\newtheorem{thm}{Theorem}[section]

\newtheorem{conj}[thm]{Conjecture}
\newtheorem{lem}[thm]{Lemma}
\newtheorem{remark}[thm]{Remark}

\renewcommand{\qed}{\hfill$\Box$\medskip}
\renewcommand{\thefootnote}{*}

\numberwithin{equation}{section}

\begin{document}
\begin{center}
{\large\bf Proof of two supercongruences by the Wilf-Zeilberger method}
\end{center}

\vskip 2mm \centerline{Guo-Shuai Mao}
\begin{center}
{\footnotesize $^1$Department of Mathematics, Nanjing
University of Information Science and Technology, Nanjing 210044,  People's Republic of China\\
{\tt maogsmath@163.com  } }
\end{center}

\vskip 0.7cm \noindent{\bf Abstract.}
In this paper, we prove two supercongruences by the Wilf-Zeilberger method. One of them is, for any prime $p>3$,
\begin{align*}
\sum_{n=0}^{(p-1)/2}\frac{3n+1}{(-8)^n}\binom{2n}n^3\equiv p\left(\frac{-1}p\right)+\frac{p^3}4\left(\frac2p\right)E_{p-3}\left(\frac14\right)\pmod{p^4},
\end{align*}
where $\left(\frac{\cdot}p\right)$ stands for the Legendre symbol, and $E_{n}(x)$ are the Euler polynomials. This congruence confirms a conjecture of Sun \cite[(2.18)]{sun-numb-2019} with $n=1$.



\renewcommand{\thefootnote}{**}

\section{Introduction}

    \qquad Recall that the Euler numbers $\{E_n\}$ and the Euler polynomials $\{E_n(x)\}$ are defined by
$$\frac{2e^t}{e^{2t}+1}=\sum_{n=0}^\infty E_n\frac{t^n}{n!}\ (|t|<\frac{\pi}2)\ \mbox{and}
\ \frac{2e^{xt}}{e^{t}+1}=\sum_{n=0}^\infty E_n(x)\frac{t^n}{n!}\ (|t|<\pi),$$
the Bernoulli numbers $\{B_n\}$ and the Bernoulli polynomials $\{B_n(x)\}$ are defined as follows:
$$\frac x{e^x-1}=\sum_{n=0}^\infty B_n\frac{x^n}{n!}\ \ (0<|x|<2\pi)\ \mbox{and}\ B_n(x)=\sum_{k=0}^n\binom nkB_kx^{n-k}\ \ (n\in\mathbb{N}).$$

In the past decade, many researchers studied supercongruences via the Wilf-Zeilberger (WZ) method (see, for instance, \cite{gl-arxiv-2019,CXH-rama-2016,he-jnt-2015,hm-rama-2017,mz-rama-2019,oz-jmaa-2016,sun-ijm-2012,zudilin-jnt-2009}). For instance, W. Zudilin \cite{zudilin-jnt-2009} proved several Ramanujan-type supercongruences by the WZ method. One of them, conjectured by van Hamme \cite{vhamme}, says that for any odd prime $p$,
\begin{align}\label{wzpr}
\sum_{k=0}^{(p-1)/2}(4k+1)(-1)^k\left(\frac{\left(\frac12\right)_k}{k!}\right)^3\equiv(-1)^{(p-1)/2}p\pmod{p^3},
\end{align}
where $(a)_n=a(a+1)\ldots(a+n-1) (n\in\{1,2,\ldots\})$ with $(a)_0=1$ is the raising factorial for $a\in\mathbb{C}$.

For $n\in\mathbb{N}$, define
$$H_n:=\sum_{0<k\leq n}\frac1k,\ H_n^{(2)}:=\sum_{0<k\leq n}\frac1{k^2},\ H_0=H_0^{(2)}=0,$$
where $H_n$ with $n\in\mathbb{N}$ are often called the classical harmonic numbers. Let $p>3$ be a prime. J. Wolstenholme \cite{wolstenholme-qjpam-1862} proved that
\begin{align}\label{hp-1}
H_{p-1}\equiv0\pmod{p^2}\ \mbox{and}\ H_{p-1}^{(2)}\equiv0\pmod p,
\end{align}
which imply that
\begin{align}
\binom{2p-1}{p-1}\equiv1\pmod{p^3}.\label{2p1p}
\end{align}

Z.-W. Sun \cite{sun-ijm-2012} proved the following supercongruence by the WZ method, for any odd prime $p$,
\begin{equation}\label{sun}
\sum_{k=0}^{p-1}\frac{4k+1}{(-64)^k}\binom{2k}k^3\equiv(-1)^{\frac{(p-1)}2}p+p^3E_{p-3}\pmod{p^4}.
\end{equation}

Guo and Liu \cite{gl-arxiv-2019} showed that for any prime $p>3$,
\begin{equation}\label{glp4}
\sum_{k=0}^{(p+1)/2}(-1)^k(4k-1)\frac{\left(-\frac12\right)_k^3}{(1)_k^3}\equiv p(-1)^{(p+1)/2}+p^3(2-E_{p-3})\pmod{p^4}.
\end{equation}
Guo also researched $q$-analogues of Ramanujan-type supercongruences and $q$-analogues of supercongruences of van Hamme (see, for instance, \cite{g-jmaa-2018,guo-rama-2019}).

Long \cite{long-2011-pjm} and Chen, Xie and He \cite{CXH-rama-2016} proved that, for any odd prime $p$,
$$
\sum_{n=0}^{(p-1)/2}\frac{6n+1}{(-512)^n}\binom{2n}n^3\equiv p\left(\frac{-2}p\right)\pmod{p^2}.
$$

Recently, the author \cite{mao-arxiv-2019} proved a conjecture of Z.-W. Sun which says that: Let $p>3$ be a prime. Then
\begin{equation*}
\sum_{n=0}^{(p-1)/2}\frac{6n+1}{(-512)^n}\binom{2n}n^3\equiv p\left(\frac{-2}p\right)+\frac{p^3}4\left(\frac2p\right)E_{p-3}\pmod{p^4}.
\end{equation*}
Chen, Xie and He \cite{CXH-rama-2016} confirmed a supercongruence conjetured by Z.-W. Sun \cite{sun-scm-2011}, which says that for any prime $p>3$,
$$
\sum_{k=0}^{p-1}\frac{3k+1}{(-8)^k}\binom{2k}k^3\equiv p(-1)^{(p-1)/2}+p^3E_{p-3}\pmod{p^4}.
$$
In this paper we prove the following result:
\begin{thm}\label{Thsun} Let $p>3$ be a prime. Then
\begin{equation}
\sum_{n=0}^{(p-1)/2}\frac{3n+1}{(-8)^n}\binom{2n}n^3\equiv p\left(\frac{-1}p\right)+\frac{p^3}{4}\left(\frac{2}p\right)E_{p-3}\left(\frac14\right)\pmod{p^4}.
\end{equation}
\end{thm}
\begin{remark}\rm This congruence confirms a conjecture of Sun \cite[(2.18)]{sun-numb-2019} with $n=1$. And this congruence with \cite[Theorem 1.5]{mao-arxiv-2019} yields that
$$
\sum_{n=0}^{(p-1)/2}\frac{3n+1}{(-8)^n}\binom{2n}n^3\equiv4\left(\frac{2}p\right)\sum_{n=0}^{p-1}\frac{6n+1}{(-512)^n}\binom{2n}n^3-3p\left(\frac{-1}p\right)\pmod{p^4},
$$
which is a conjecture of Sun \cite[Conjecture 5.1]{sun-scm-2011}.
\end{remark}
Guo \cite{g-jmaa-2018} proved that
$$
\sum_{k=0}^{(p^r-1)/2}\frac{4k+1}{(-64)^k}\binom{2k}k^3\equiv(-1)^{\frac{(p-1)r}2}p^r\pmod{p^{r+2}},
$$
and in the same paper he proposed a conjecture as follows:
\begin{conj}\label{Guo-2018}{\rm (\cite[Conjecture 5.1]{g-jmaa-2018})}
$$
\sum_{k=0}^{p^r-1}\frac{4k+1}{(-64)^k}\binom{2k}k^3\equiv(-1)^{\frac{(p-1)r}2}p^r\pmod{p^{r+2}}.
$$
\end{conj}
Guo and zudilin proved this Conjecture by founding its $q$-anology, (see \cite{GZ-aim-2019}). And the author \cite{mao-2019-arxiv} gave a new proof of it by the WZ method.
\begin{thm}\label{Th10ngz} For any prime $p>3$ and integer $r>0$, we have
\begin{align}
\sum_{n=0}^{p^r-1}\frac{3n+1}{(-8)^n}\binom{2n}n^3&\equiv(-1)^{(p^r-1)/2}p^r\pmod{p^{r+2}}.
\end{align}
\end{thm}
Our main tool in this paper is the WZ method. We shall prove Theorem \ref{Thsun} in the next Section. And Theorem \ref{Th10ngz} will be proved in Section 3.
\section{Proof of Theorem \ref{Thsun}}
\qquad We will use the following WZ pair which appears in \cite{CXH-rama-2016} to prove Theorem \ref{Thsun}. For nonnegative integers $n, k$, define
$$
F(n,k)=\frac{(-1)^{n}(3n-2k+1)\binom{2n}n\binom{2n-2k}{n-k}\binom{2n-2k}{n}}{2^{3n-2k}}
$$
and
$$
G(n,k)=\frac{(-1)^{n+1}n\binom{2n}n\binom{2n-2k}{n-k}\binom{2n-2k}{n-1}}{2^{3n-2k}}.
$$
Clearly $F(n,k)=G(n,k)=0$ if $n<k$ and $F(n,n)=0$ for any positive integer $n$. It is easy to check that
\begin{equation}\label{FG}
F(n,k-1)-F(n,k)=G(n+1,k)-G(n,k)
\end{equation}
for all nonnegative integer $n$ and $k>0$.

Summing (\ref{FG}) over $n$ from $0$ to $(p-1)/2$ we have
$$
\sum_{n=0}^{(p-1)/2}F(n,k-1)-\sum_{n=0}^{(p-1)/2}F(n,k)=G\left(\frac{p+1}2,k\right)-G(0,k)=G\left(\frac{p+1}2,k\right).
$$
Furthermore, summing both side of the above identity over $k$ from $1$ to $(p-1)/2$, we obtain
\begin{align}\label{wz1}
\sum_{n=0}^{(p-1)/2}F(n,0)=\sum_{k=1}^{(p-1)/2}G\left(\frac{p+1}2,k\right).
\end{align}
\begin{lem}\label{mor}{\rm (\cite{Mor})} For any prime $p>3$, we have
$$
\binom{p-1}{(p-1)/2}\equiv(-1)^{(p-1)/2}4^{p-1}\pmod{p^3}.
$$
\end{lem}
By the definition of $G(n,k)$ we have
\begin{align*}
G\left(\frac{p+1}2,k\right)&=(-1)^{\frac{p-1}2}\frac{\frac{p+1}2\binom{p+1}{\frac{p+1}2}\binom{p+1-2k}{\frac{p+1}2-k}\binom{p+1-2k}{\frac{p-1}2}}{2^{\frac{3p+3}2-2k}}\notag\\
&=\frac{2p(-1)^{\frac{p-1}2}\binom{p-1}{\frac{p-1}2}}{2^{\frac{3p+3}2}}\binom{p+1-2k}{\frac{p+1}2-k}\binom{p+1-2k}{\frac{p-1}2}4^k.
\end{align*}
It is easy to see that
$$
\frac{\binom{2n-2k}{n-k}}{4^{n-k}}=\frac{\left(\frac12\right)_{n-k}}{(n-k)!}, \ \ \left(\frac12\right)_{n-k}\left(\frac12+n-k\right)_{k-1}=\left(\frac12\right)_{n-1}
$$
and
$$
\binom{2n-2k}{n-1}=\frac{(2n-2k)!}{(n-1)!(n+1-2k)!}=\frac{\binom{2n-2k}{n-k}(n-k)!^2}{(n-1)!(n+1-2k)!}.
$$
So we have
\begin{align}\label{Gpk}
G\left(\frac{p+1}2,k\right)&=\frac{(-1)^{\frac{p-1}2}2p\binom{p-1}{\frac{p-1}2}4^{p+1}}{\left(\frac{p-1}2\right)!2^{\frac{3p+3}2}}\frac{\left(\frac12\right)_{(p+1)/2-k}^2}{\left(\frac{p+3}2-2k\right)!4^k}\notag\\
&=\frac{32p(-1)^{\frac{p-1}2}\binom{p-1}{\frac{p-1}2}^3}{2^{\frac{3p+3}2}}\frac{\left(\frac{p-1}2\right)!}{\left(\frac{p+3}2-2k\right)!\left(\frac p2+1-k\right)_{k-1}^24^k}.
\end{align}
\begin{lem}\label{p-12k} Let $p>3$ be a prime. Then
$$
2^{\frac{9p-9}2}\sum_{k=0}^{\frac{p-3}2}\frac{\binom{\frac{p-1}2}{2k}\binom{2k}k}{4^k}\equiv (-1)^{(p-1)/2}\left(1+6pq_p(2)+15p^2q_p(2)^2\right)\pmod{p^3}.
$$
\end{lem}
\begin{proof}
First we have the following two identities:
$$
\sum_{k=0}^n\frac{\binom{2n}k\binom{2n-k}k}{4^k}=\frac{\binom{4n}{2n}}{4^n}\ \ \mbox{and}\ \ \sum_{k=0}^n\frac{\binom{2n+1}k\binom{2n+1-k}k}{4^k}=\frac{\binom{4n+1}{2n+1}}{4^n}.
$$
So when $p\equiv1\pmod4$, by Lemma \ref{mor} we have
\begin{align*}
2^{\frac{9p-9}2}\sum_{k=0}^{\frac{p-3}2}\frac{\binom{\frac{p-1}2}{2k}\binom{2k}k}{4^k}&=2^{\frac{9p-9}2}\sum_{k=0}^{\frac{p-1}4}\frac{\binom{\frac{p-1}2}{k}\binom{\frac{p-1}2-k}k}{4^k}=2^{4p-4}\binom{p-1}{\frac{p-1}2}\equiv(-1)^{\frac{p-1}2}2^{6p-6}\\
&\equiv(-1)^{(p-1)/2}\left(1+6pq_p(2)+15p^2q_p(2)^2\right)\pmod{p^3}.
\end{align*}
And when $p\equiv3\pmod4$, we have
\begin{align*}
2^{\frac{9p-9}2}\sum_{k=0}^{\frac{p-3}2}\frac{\binom{\frac{p-1}2}{2k}\binom{2k}k}{4^k}&=2^{\frac{9p-9}2}\sum_{k=0}^{\frac{p-3}4}\frac{\binom{\frac{p-1}2}{k}\binom{\frac{p-1}2-k}k}{4^k}=2^{4p-3}\binom{p-2}{\frac{p-1}2}\equiv(-1)^{\frac{p-1}2}2^{6p-6}\\
&\equiv(-1)^{(p-1)/2}\left(1+6pq_p(2)+15p^2q_p(2)^2\right)\pmod{p^3}.
\end{align*}
Therefore the proof of Lemma \ref{p-12k} is finished.
\end{proof}
\begin{lem}\label{p-12khk} For any prime $p>3$, we have
$$
2^{\frac{9p-9}2}\sum_{k=0}^{\frac{p-3}2}\frac{\binom{\frac{p-1}2}{2k}\binom{2k}kH_k}{4^k}\equiv-3(-1)^{\frac{p-1}2}\left(2q_p(2)+11pq_p(2)^2\right)\pmod{p^2}.
$$
\end{lem}
\begin{proof}
First we have the following two identities:
$$
\sum_{k=0}^n\frac{\binom{2n}k\binom{2n-k}kH_k}{4^k}=\frac{\binom{4n}{2n}}{4^n}(3H_{2n}-2H_{4n}),$$
$$
\sum_{k=0}^n\frac{\binom{2n+1}k\binom{2n+1-k}kH_k}{4^k}=\frac{\binom{4n+1}{2n+1}}{4^n}\left(3H_{2n+1}-2H_{4n+2}\right).
$$
So when $p\equiv1\pmod4$, by Lemma \ref{mor}, \cite[Theorem 3.2]{sun-jnt-2008} and (\ref{hp-1}) we have
\begin{align*}
&2^{\frac{9p-9}2}\sum_{k=0}^{\frac{p-3}2}\frac{\binom{\frac{p-1}2}{2k}\binom{2k}kH_k}{4^k}=2^{\frac{9p-9}2}\sum_{k=0}^{\frac{p-1}4}\frac{\binom{\frac{p-1}2}{k}\binom{\frac{p-1}2-k}kH_k}{4^k}=2^{4p-4}\binom{p-1}{\frac{p-1}2}(3H_{\frac{p-1}2}-2H_{p-1})\\
&\equiv3(-1)^{\frac{p-1}2}2^{6p-6}H_{\frac{p-1}2}\equiv3(-1)^{(p-1)/2}\left(1+6pq_p(2)\right)(-2q_p(2)+pq_p(2))\\
&\equiv-3(-1)^{\frac{p-1}2}\left(2q_p(2)+11pq_p(2)^2\right)\pmod{p^2}.
\end{align*}
And when $p\equiv3\pmod4$, we have
\begin{align*}
&2^{\frac{9p-9}2}\sum_{k=0}^{\frac{p-3}2}\frac{\binom{\frac{p-1}2}{2k}\binom{2k}kH_k}{4^k}=2^{\frac{9p-9}2}\sum_{k=0}^{\frac{p-3}4}\frac{\binom{\frac{p-1}2}{k}\binom{\frac{p-1}2-k}kH_k}{4^k}=2^{4p-4}\binom{p-1}{\frac{p-1}2}(3H_{\frac{p-1}2}-2H_{p-1})\\
&\equiv3(-1)^{\frac{p-1}2}2^{6p-6}H_{\frac{p-1}2}\equiv3(-1)^{(p-1)/2}\left(1+6pq_p(2)\right)(-2q_p(2)+pq_p(2))\\
&\equiv-3(-1)^{\frac{p-1}2}\left(2q_p(2)+11pq_p(2)^2\right)\pmod{p^2}.
\end{align*}
Therefore the proof of Lemma \ref{p-12khk} is completed.
\end{proof}
\begin{lem}\label{p-12khkhk2} For any prime $p>3$, we have
$$
2^{\frac{9p-9}2}\sum_{k=0}^{\frac{p-3}2}\frac{\binom{\frac{p-1}2}{2k}\binom{2k}k\left(H_k^2+H_{k}^{(2)}\right)}{4^k}\equiv36(-1)^{\frac{p-1}2}q_p(2)^2\pmod{p}.
$$
\end{lem}
\begin{proof}
First we have the following two identities:
$$
\sum_{k=0}^n\frac{\binom{2n}k\binom{2n-k}k\left(H_k^2+H_k^{(2)}\right)}{4^k}=\frac{\binom{4n}{2n}}{4^n}\left(5H_{2n}^{(2)}-4H_{4n}^{(2)}\right)+\frac{\binom{4n}{2n}}{4^n}\left(3H_{2n}-2H_{4n}\right)^2,$$
$$
\sum_{k=0}^n\frac{\binom{2n+1}k\binom{2n+1-k}k\left(H_k^2+H_k^{(2)}\right)}{4^k}=\frac{\binom{4n+1}{2n+1}}{4^n}\left(\left(5H_{2n+1}^{(2)}-4H_{4n+2}^{(2)}\right)+\left(3H_{2n+1}-2H_{4n+2}\right)^2\right).
$$
So when $p\equiv1\pmod4$, by Lemma \ref{mor}, \cite[Theorem 3.5]{sun-jnt-2008} and (\ref{hp-1}) we have
\begin{align*}
&2^{\frac{9p-9}2}\sum_{k=0}^{\frac{p-3}2}\frac{\binom{\frac{p-1}2}{2k}\binom{2k}k\left(H_k^2+H_k^{(2)}\right)}{4^k}=2^{\frac{9p-9}2}\sum_{k=0}^{\frac{p-1}4}\frac{\binom{\frac{p-1}2}{k}\binom{\frac{p-1}2-k}k\left(H_k^2+H_k^{(2)}\right)}{4^k}\\
&=2^{4p-4}\binom{p-1}{\frac{p-1}2}\left(\left(5H_{(p-1)/2}^{(2)}-4H_{p-1}^{(2)}\right)+\left(3H_{(p-1)/2}-2H_{p-1}\right)^2\right)\\
&\equiv(-1)^{\frac{p-1}2}9H_{\frac{p-1}2}^2\equiv36(-1)^{(p-1)/2}q_p(2)^2\pmod{p}.
\end{align*}
And when $p\equiv3\pmod4$, we have
\begin{align*}
&2^{\frac{9p-9}2}\sum_{k=0}^{\frac{p-3}2}\frac{\binom{\frac{p-1}2}{2k}\binom{2k}k\left(H_k^2+H_k^{(2)}\right)}{4^k}=2^{\frac{9p-9}2}\sum_{k=0}^{\frac{p-3}4}\frac{\binom{\frac{p-1}2}{k}\binom{\frac{p-1}2-k}k\left(H_k^2+H_k^{(2)}\right)}{4^k}\\
&=2^{4p-4}\binom{p-1}{\frac{p-1}2}\left(\left(5H_{(p-1)/2}^{(2)}-4H_{p-1}^{(2)}\right)+\left(3H_{(p-1)/2}-2H_{p-1}\right)^2\right)\\
&\equiv(-1)^{\frac{p-1}2}9H_{\frac{p-1}2}^2\equiv36(-1)^{(p-1)/2}q_p(2)^2\pmod{p}.
\end{align*}
Therefore the proof of Lemma \ref{p-12khkhk2} is complete.
\end{proof}
\begin{lem}\label{mos}{\rm [See \cite{MOS}]} Let $x$ and $y$ be variables and $n\in\mathbb{N}$. Then
\begin{align*}
&{\rm (i)}.\  B_{2n+1}=0.\\
&{\rm (ii)}.\ B_n(1-x)=(-1)^nB_n(x).\\
&{\rm (iii)}.\  B_n(x+y)=\sum_{k=0}^n\binom{n}kB_{n-k}(y)x^k.\\
&{\rm (iv)}.\  E_{n-1}(x)=\frac{2^n}n\left(B_n\left(\frac{x+1}2\right)-B_n\left(\frac x2\right)\right).
\end{align*}
\end{lem}
\begin{lem}\label{p-12khk2} For any prime $p>3$, we have
$$
\sum_{k=0}^{\frac{p-3}2}\frac{\binom{\frac{p-1}2}{2k}\binom{2k}kH_{k}^{(2)}}{4^k}\equiv\sum_{k=1}^{\lfloor\frac{p-1}4\rfloor}\frac{\binom{4k}{2k}\binom{2k}kH_{k}^{(2)}}{{64}^k}\equiv-E_{p-3}\left(\frac14\right)\pmod{p}.
$$
\end{lem}
\begin{proof} It is easy to see that
$$
\sum_{k=0}^{\frac{p-3}2}\frac{\binom{\frac{p-1}2}{2k}\binom{2k}kH_{k}^{(2)}}{4^k}=\sum_{k=0}^{\lfloor\frac{p-1}4\rfloor}\frac{\binom{\frac{p-1}2}{2k}\binom{2k}kH_{k}^{(2)}}{4^k}\equiv\sum_{k=1}^{\lfloor\frac{p-1}4\rfloor}\frac{\binom{4k}{2k}\binom{2k}kH_{k}^{(2)}}{{64}^k}\pmod p
$$
since $\binom{\frac{p-1}2}{2k}\equiv\binom{4k}{2k}/{16}^k\pmod p$ for each $k\in\{0,1,\ldots,\lfloor\frac{p-1}4\rfloor\}$ and $H_0^{(2)}=0$. So we just need to verify that
$$
\sum_{k=1}^{\lfloor\frac{p-1}4\rfloor}\frac{\binom{4k}{2k}\binom{2k}kH_{k}^{(2)}}{{64}^k}\equiv-E_{p-3}\left(\frac14\right)\pmod{p}.
$$
We have three identities:
$$
\sum_{k=1}^n\binom{n}k\binom{-3/4}kH_k^{(2)}=(-1)^n\binom{-1/4}n\left(H_n^{(2)}-\sum_{k=1}^n\frac{(-1)^k}{k^2\binom{-1/4}k}\right),
$$
$$
\sum_{k=1}^n\binom{n}k\binom{-1/4}kH_k^{(2)}=(-1)^n\binom{-3/4}n\left(H_n^{(2)}-\sum_{k=1}^n\frac{(-1)^k}{k^2\binom{-3/4}k}\right)
$$
and
$$
\sum_{k=1}^n\frac{(-1)^k}{k^2\binom{n}k}=H_n^{(2)}+2\sum_{k=1}^n\frac{(-1)^k}{k^2}.
$$
It is known that $\frac{\binom{4k}{2k}\binom{2k}k}{{64}^k}=\binom{-1/4}k\binom{-3/4}k$. Thus, if $p\equiv1\pmod4$, then
\begin{align*}
\sum_{k=1}^{\lfloor\frac{p-1}4\rfloor}\frac{\binom{4k}{2k}\binom{2k}kH_{k}^{(2)}}{{64}^k}&\equiv\sum_{k=1}^{\frac{p-1}4}\binom{\frac{p-1}4}k\binom{-3/4}kH_k^{(2)}\equiv(-1)^{\frac{p-1}4}\left(H_{\frac{p-1}4}^{(2)}-\sum_{k=1}^{\frac{p-1}4}\frac{(-1)^k}{k^2\binom{\frac{p-1}4}k}\right)\\
&=-2(-1)^{\frac{p-1}4}\sum_{k=1}^{\frac{p-1}4}\frac{(-1)^k}{k^2}\pmod p.
\end{align*}
If $p\equiv3\pmod 4$, then
\begin{align*}
\sum_{k=1}^{\lfloor\frac{p-1}4\rfloor}\frac{\binom{4k}{2k}\binom{2k}kH_{k}^{(2)}}{{64}^k}&\equiv\sum_{k=1}^{\frac{p-3}4}\binom{\frac{p-3}4}k\binom{-1/4}kH_k^{(2)}\equiv(-1)^{\frac{p-3}4}\left(H_{\frac{p-3}4}^{(2)}-\sum_{k=1}^{\frac{p-1}4}\frac{(-1)^k}{k^2\binom{\frac{p-3}4}k}\right)\\
&=-2(-1)^{\frac{p-3}4}\sum_{k=1}^{\frac{p-3}4}\frac{(-1)^k}{k^2}\pmod p.
\end{align*}
Hence
$$
\sum_{k=1}^{\lfloor\frac{p-1}4\rfloor}\frac{\binom{4k}{2k}\binom{2k}kH_{k}^{(2)}}{{64}^k}\equiv-2(-1)^{\lfloor\frac{p-1}4\rfloor}\sum_{k=1}^{\lfloor\frac{p-1}4\rfloor}\frac{(-1)^k}{k^2}\pmod p.
$$
In view of \cite[Lemma2.1]{sun-jnt-2008}, for any $p,m\in\mathbb{N}$ and $k,r\in\mathbb{Z}$ with $k\geq0$, we have
$$
\sum_{\substack{x=0\\{x\equiv r\pmod{m}}}}^{p-1}x^k=\frac{m^k}{k+1}\left(B_{k+1}\left(\frac{p}{m}+\left\{\frac{r-p}{m}\right\}\right)-B_{k+1}\left(\left\{\frac{r}{m}\right\}\right)\right).
$$
By using this identity, (i) and (iii) of Lemma \ref{mos}, we have
\begin{align*}
\sum_{j=1}^{\lfloor\frac{p-1}8\rfloor}\frac1{j^2}=64\sum_{\substack{k=1\\k\equiv0\pmod8}}^{p-1}\frac{1}{k^2}\equiv64\sum_{\substack{x=0\\x\equiv0\pmod8}}^{p-1}x^{p-3}\equiv-\frac12B_{p-2}\left(\left\{-\frac p8\right\}\right)\pmod p
\end{align*}
and
\begin{align*}
\sum_{j=1}^{\lfloor\frac{p-1}8\rfloor}\frac1{(2j-1)^2}&=16\sum_{j=1}^{\lfloor\frac{p-1}8\rfloor}\frac1{(8j-4)^2}=16\sum_{\substack{j=1\\j\equiv4\pmod8}}^{p-1}\frac1{j^2}\equiv16\sum_{\substack{x=0\\j\equiv4\pmod8}}^{p-1}x^{p-3}\\
&\equiv-\frac18B_{p-2}\left(\left\{\frac{4-p}8\right\}\right)\pmod p.
\end{align*}
So
\begin{align*}
\sum_{k=1}^{\lfloor\frac{p-1}4\rfloor}\frac{(-1)^k}{k^2}&=\frac14\sum_{k=1}^{\lfloor\frac{p-1}8\rfloor}\frac{1}{k^2}-\sum_{k=1}^{\lfloor\frac{p-1}8\rfloor}\frac{1}{(2k-1)^2}\\
&\equiv\frac18\left(B_{p-2}\left(\left\{\frac{4-p}8\right\}\right)-B_{p-2}\left(\left\{-\frac p8\right\}\right)\right)\pmod p.
\end{align*}
Hence
$$
\sum_{k=1}^{\lfloor\frac{p-1}4\rfloor}\frac{\binom{4k}{2k}\binom{2k}kH_{k}^{(2)}}{{64}^k}\equiv-\frac14(-1)^{\lfloor\frac{p-1}4\rfloor}\left(B_{p-2}\left(\left\{\frac{4-p}8\right\}\right)-B_{p-2}\left(\left\{-\frac p8\right\}\right)\right)\pmod p.
$$
Then we immediately get the desired result by (iv) of Lemma \ref{mos}.
\end{proof}
\begin{lem} \label{G12} For any primes $p>3$, we have
$$
\sum_{k=1}^{(p-1)/2}G\left(\frac{p+1}2,k\right)\equiv p\left(\frac{-1}p\right)+\frac{p^3}4\left(\frac{2}p\right)E_{p-3}\left(\frac14\right)\pmod{p^4}.
$$
\end{lem}
\begin{proof} It is easy to check that
\begin{align*}
\left(\frac p2+1-k\right)_{k-1}^2&=\left(\frac p2+1-k\right)^2\ldots\left(\frac p2-1\right)^2\\
&\equiv(k-1)!^2\left(1-pH_{k-1}+\frac{p^2}4\left(2H_{k-1}^2-H_{k-1}^{(2)}\right)\right)\pmod{p^3}.
\end{align*}
This with (\ref{Gpk}) yields that, modulo $p^4$ we have
\begin{align*}
G\left(\frac{p+1}2,k\right)&\equiv\frac{32p(-1)^{\frac{p-1}2}\binom{p-1}{\frac{p-1}2}^3}{2^{\frac{3p+3}2}}\frac{\binom{\frac{p-1}2}{2k-2}\binom{2k-2}{k-1}}{\left(1-pH_{k-1}+\frac{p^2}2H_{k-1}^2-\frac{p^2}4H_{k-1}^{(2)}\right)4^k}\\
&\equiv\frac{32p(-1)^{\frac{p-1}2}\binom{p-1}{\frac{p-1}2}^3}{2^{\frac{3p+3}2}}\frac{\binom{\frac{p-1}2}{2k-2}\binom{2k-2}{k-1}}{4^k}\left(1+pH_{k-1}+\frac{p^2}2H_{k-1}^2+\frac{p^2}4H_{k-1}^{(2)}\right).
\end{align*}
So by Lemma \ref{mor}, we have
\begin{align*}
\sum_{k=1}^{(p-1)/2}G\left(\frac{p+1}2,k\right)&\equiv 4p2^{\frac{9p-9}2}\sum_{k=1}^{\frac{p-1}2}\frac{\binom{\frac{p-1}2}{2k-2}\binom{2k-2}{k-1}}{4^k}\left(1+pH_{k-1}+\frac{p^2}2H_{k-1}^2+\frac{p^2}4H_{k-1}^{(2)}\right)\\
&=p2^{\frac{9p-9}2}\sum_{k=0}^{\frac{p-3}2}\frac{\binom{\frac{p-1}2}{2k}\binom{2k}{k}}{4^k}\left(1+pH_{k}+\frac{p^2}2H_{k}^2+\frac{p^2}4H_{k}^{(2)}\right)\pmod{p^4}.
\end{align*}
By Lemmas \ref{p-12k}--\ref{p-12khkhk2}, \ref{p-12khk2} and
$$
2^{(p-1)/2}\equiv\left(\frac2p\right)\left(1+\frac p2q_p(2)-\frac{p^2}8q_p(2)^2\right)\pmod {p^3},
$$
we immediately obtain the desired result of Lemma \ref{G12}.
\end{proof}
\noindent{\it Proof of Theorem \ref{Thsun}}. Combining (\ref{wz1}) with Lemma \ref{G12}, we immediately get that for any prime $p>3$
$$
\sum_{n=0}^{(p-1)/2}\frac{3n+1}{(-8)^n}\binom{2n}n^3\equiv p\left(\frac{-1}p\right)+\frac{p^3}4\left(\frac{2}p\right)E_{p-3}\left(\frac14\right)\pmod{p^4}.
$$
Therefore the proof of Theorem \ref{Thsun} is complete. \qed
\section{Proof of Theorem \ref{Th10ngz}}
By the same WZ pair in Section 2, we have
\begin{align}\label{wz2}
\sum_{n=0}^{p^r-1}F(n,0)=\sum_{k=1}^{p^r-1}G\left(p^r,k\right).
\end{align}
\begin{lem}\label{smor}{\rm (\cite[Lemma 2.4]{long-2011-pjm})} For any prime $p>3$, we have
$$
\binom{p^r-1}{(p^r-1)/2}\equiv(-1)^{(p^r-1)/2}4^{p^r-1}\pmod{p^3}.
$$
\end{lem}

\begin{lem}\label{Gprpr+12} For any prime $p>3$ and integer $r>0$, we have
$$
G\left(p^r,\frac{p^r+1}2\right)\equiv(-1)^{(p^r-1)/2}p^r\pmod{p^{r+2}}.
$$
\end{lem}
\begin{proof} By the definition of $G(n,k)$, we have
$$G\left(p^r,\frac{p^r+1}2\right)=\frac{p^r\binom{2p^r}{p^r}\binom{p^r-1}{(p^r-1)/2}}{2^{2p^r-1}}.$$
It is easy to see that
\begin{align}\label{mao}
\binom{2p^r}{p^r}=2\prod_{k=1}^{p^r-1}\left(1-\frac{2p^r}k\right)\equiv2-4p^rH_{p^r-1}\equiv2-4pH_{p-1}\equiv2\pmod {p^2}
\end{align}
with Wolstenholme's result $H_{p-1}\equiv0\pmod{p^2}$ as mentioned in the introduction.

So by Lemma \ref{smor} and (\ref{mao}), we have
$$
G\left(p^r,\frac{p^r+1}2\right)\equiv(-1)^{(p^r-1)/2}p^r\pmod{p^{r+2}}.
$$
Now the proof of Lemma \ref{Gprpr+12} is complete.
\end{proof}
\begin{lem}\label{Gprk} For any prime $p>3$ and integer $r>0$, we have
$$
\sum_{k=1}^{(p^r-1)/2}G\left(p^r,k\right)\equiv0\pmod{p^{r+2}}.
$$
\end{lem}
\begin{proof} By the definition of $G(n,k)$, we have
\begin{align*}
\sum_{k=1}^{(p^r-1)/2}G\left(p^r,k\right)&=\frac{p^r\binom{2p^r}{p^r}}{2^{3p^r}}\sum_{k=1}^{(p^r-1)/2}\binom{2p^r-2k}{p^r-k}\binom{2p^r-2k}{p^r-1}4^k\\
&=\frac{p^r\binom{2p^r}{p^r}}{2^{3p^r}}\sum_{k=1}^{(p^r-1)/2}\binom{2p^r-2k}{p^r-k}\binom{-p^r}{p^r+1-2k}4^k\\
&=-\frac{p^{2r}\binom{2p^r}{p^r}}{2^{3p^r}}\sum_{k=1}^{(p^r-1)/2}\binom{2p^r-2k}{p^r-k}\binom{-p^r-1}{p^r-2k}\frac{4^k}{p^r+1-2k},
\end{align*}
where we used $\binom{n}{k}=\binom{n}{n-k}$ and $\binom{n}{k}=(-1)^k\binom{-n+k-1}k$.

In view of the paper \cite{PS}, let $k$ and $l$ be positive integers with $k+l=p^r$ and $0<l<p^r/2$, we have
$$
l\binom{2l}l\binom{2k}k\equiv-2p^r\pmod{p^{r+1}},\ \ \frac{-2p^r}{l\binom{2l}l}\equiv\binom{2k}k\pmod{p^{2}}
$$
and
\begin{align}\label{2kk}
\binom{2p^r-2l}{p^r-l}\equiv0\pmod{p}.
\end{align}
Thus,
$$
\sum_{k=1}^{(p^r-1)/2}G\left(p^r,k\right)\equiv0\pmod{p^{r+2}},
$$
since
$$
\binom{-p^r-1}{p^r-2k}=-\prod_{j=1}^{p^r-2k}\left(1+\frac{p^r}j\right)\equiv1\pmod p
$$
and
$$
\binom{2p^r-2k}{p^r-k}\equiv0\pmod p,\ \ \ \ \frac{p^r}{p^r+1-2k}\equiv0\pmod p
$$
for each $1\leq k\leq(p^r-1)/2$.

Therefore the proof of Lemma \ref{Gprk} is completed.
\end{proof}
\noindent{\it Proof of Theorem \ref{Th10ngz}}. By the definition of $G(n,k)$, we have
$$
\sum_{k=1}^{p^r-1}G\left(p^r,k\right)=\sum_{k=1}^{(p^r+1)/2}G\left(p^r,k\right).
$$
So combining (\ref{wz2}), Lemmas \ref{Gprpr+12} and \ref{Gprk}, we immediately obtain that for any prime $p>3$ and integer $r>0$,
$$
\sum_{n=0}^{p^r-1}\frac{3n+1}{(-8)^n}\binom{2n}n^3\equiv(-1)^{(p^r-1)/2}p^r\pmod{p^{r+2}}.
$$
Therefore the proof of Theorem \ref{Th10ngz} is complete.\qed

\vskip 3mm \noindent{\bf Acknowledgments.}
The author is funded by the Startup Foundation for Introducing Talent of Nanjing University of Information Science and Technology (2019r062).

\end{document}